\documentclass[12pt,a4paper]{amsart}
\usepackage{amsmath}
\usepackage{amsthm}
\usepackage{amssymb}
\usepackage{mathrsfs}
\usepackage[english]{babel}
\usepackage{ot1patch}
\usepackage{graphicx}
\usepackage{comment}
\usepackage{csquotes}

\usepackage[style=numeric,giveninits=true]{biblatex}
\newtheorem{thm}{Theorem}[section]

\newtheorem{cor}[thm]{Corollary}
\newtheorem{prop}[thm]{Proposition}

\theoremstyle{remark}
\newtheorem{rem}[thm]{Remark}
\newtheorem*{rem*}{Remark}

\theoremstyle{definition}
\newtheorem{dfn}[thm]{Definition}
\newtheorem{ex}[thm]{Example}

\numberwithin{equation}{section}

\title{On the singular points approached by the medial axis.}
\author{by Adam Bialozyt}

\begin{document}
\begin{abstract}
This paper studies the singular points of a set reached by the medial axis. The investigation of $\mathscr{C}^1$ smooth points generalises the results of Birbrair and Denkowski \cite{BirbrairDenkowski} into higher dimensions and general o-minimal structures. The set of points away from the set's boundary where the set fails to be $\mathscr{C}^1$ smooth is also studied.
\end{abstract}

\maketitle

\section{Introduction}

The medial axis of a closed set $X\subset\mathbb{R}^n$ (cf. Definition~\ref{definicja}), introduced by Blum in \cite{Blum} as a locus of points in $\mathbb{R}^n\backslash X$ with more than one closest point in $X$ plays a central object in the pattern recognition theory. It emerges under various names in numerous mathematical and application problems. Its ability of lossless data compression makes it an appealing object in tomography, robotics, and simulations. On the other hand, variants of its definition appear naturally in the fields of partial differential equations, convex analysis and others. An inherent connection with the initial set's geometry makes it an interesting object for investigating geometrical and topological properties, such as its singularities or homotopy groups.

In this paper, we focus our attention on the geometrical study of relations between the medial axis and the singularities of $X$. The investigation is split into two parts. After the preliminaries of the second section, we study the $\mathcal{C}^1$ smooth part of $X$ approached by its medial axis in the third section. In Theorem~\ref{superquadratic theorem}, we prove that the medial axis~$M_X$ detects points of $X$ where the set deviates quickly from its tangent cone. Theorem~\ref{Superqadratic curve} inverts this result for curves. The fourth section concentrates on $\mathcal{C}^1$ singular part of $X$. In Theorem~\ref{lepszy motzkin}, we propose a criterium for detecting a medial axis approaching the set based on the tangent cones of $X$. Another one, based on finding an appropriate submanifold in $X$, is presented in Corollary~\ref{ostatni wniosek}.  

The first historical result in the direction of singularities detection by the medial axis, although never explicitly stated as a medial axis fact, is a revered theorem by John Nash \cite{Nash}. 
\begin{thm}[Nash Lemma]\label{Nash}
Let $X$ be a $\mathscr{C}^k$-submanifold of $\mathbb{R}^n$ with $k\geq 2$.
Then there exists $U$ - a neighbourhood of $X$ such that 
\begin{enumerate}
    \item For all $a\in U$, there is precisely one point of $X$ closest to $a$;
    \item The function assigning to a point an element in $X$ closest to it is $\mathscr{C}^{k-1}$ smooth.
\end{enumerate}
\end{thm}
Knowing that the set of discontinuities of the function assigning to a point an element in $X$ closest to it coincides with the medial axis of $X$,  Theorem~\ref{Nash} yields an immediate consequence: the medial axis of $X$ is separated from the $\mathcal{C}^k$-smooth part of $X$ for $k\geq 2$. The proof of the Nash Lemma uses the implicit function theorem applied to the partial derivatives of a parametrisation of $X$, so the assumption on $k$ cannot be lowered. This raises a natural scientific curiosity about the cases $k=1$ and $k=0$. The question is also most natural from the point of view of partial differential equations theory. The distance function is a viscosity solution for a simple eikonal equation $\|\nabla d\|=1$ with the zero Dirichlet condition \cite{Lions}. While the equation order equals one, it is natural to ask about the largest domain where the solution is a $\mathscr{C}^1$-smooth function. The Nash Lemma settles then the existence of $\mathscr{C}^1$-smooth solution for the regions with at least $\mathscr{C}^2$-smooth boundary, the $\mathscr{C}^2$-singular boundaries, however, remain out of its grasp. Another motivation behind studying the intersection of a set and its medial axis originates from the fascinating Steiner formula. According to his results, whenever the medial axis does not approach a given set, the volume of the set's parallelogram can be expressed by a polynomial \cite{Federer}.

Throughout the paper, we restrict our attention to sets that are definable in some o-minimal structure, expanding the field of real numbers. Such an approach gives us a framework with appealing definitions of global and local dimensions (coinciding with the Hausdorff dimension) and a handful of valuable tools such as curve selection lemma and cylindrical definable cell decomposition. At the same time, it protects us from pathological sets while conserving the applicability of the setting. Readers who are not familiar with the notion of definable sets may think of them as semialgebraic sets. An excellent introduction to the notion is found in \cite{Coste} or \cite{Rolin}. Although we do not require any more conditions on the o-minimal structure, on a few occasions, we compare our results with sets definable in polynomially bounded structures. These are the structures where the growth of any definable function (that is, a function with a definable graph) can be bounded by a polynomial. Again, the structure of semialgebraic sets forms an example of a polynomially bounded structure.

\section{Preliminaries}
In this section, we define basic objects which are of our interest and recall their basic properties.

For $x,y\in\mathbb{R}^n$, we denote by $[x,y]$ the closed segment joining $x$ and $y$, by $d(x,y)$ their euclidean distance, and by $\angle(x,y)$ the angle formed by these vectors, provided they are nonzero. The closed ball centred at~$a$, of radius $r$, is denoted by $\mathbb{B}(a,r)$, and $\mathbb{S}(a,r)$ denotes its boundary -- an $(n-1)$-dimensional sphere of radius $r$ centred at $a$. We adopt also Minkowski notation for dilations -- for a pair of a set and a point $(X,v)$ we denote  $X+v:=\{x+v\mid x\in X\}$, $Xv:=\{xv\mid x\in X\}$.

Whenever in the paper the continuity (or upper- and lower limits) of a family of sets or a (multi-)function is mentioned, it refers to the continuity (or upper- and lower limits) in the Kuratowski sense. Let us recall quickly that being given a family $\lbrace X_t\rbrace_{t\in\mathbb{R}^k}$ of subsets of $\mathbb{R}^n$, a point $x\in \mathbb{R}^n$ belongs to:
\[\limsup_{t\rightarrow t_0} X_t\text{ iff there exist sequences }\mathbb{R}^k\ni t_\nu\rightarrow t_0\text{ and }X_{t_\nu}\ni x_\nu\rightarrow x;\]
\[\liminf_{t\rightarrow t_0} X_t\text{ iff for each sequence }\mathbb{R}^k\ni t_\nu\rightarrow t_0\text{ we can find }X_{t_\nu}\ni x_\nu\rightarrow x.\]
Naturally, whenever two limiting sets coincide, we call them the Kuratowski limit. More on the Kuratowski convergence is found in the book \cite{RockafellarWets}, and an introduction to its relation with medial axes and conflict sets is given in \cite{DenkowskiLimits}. 

We will use the notation $Reg_k X$ for the points of $X$ at which it is a $\mathscr{C}^k$-submanifold and $Sng_kX:=X\setminus Reg_k X$. When $X$ is definable in some o-minimal structure, $Reg_k X$ forms a dense definable subset of $X$.

For a closed nonempty subset $X$ of $\mathbb{R}^n$ endowed with the euclidean norm, we define the distance of a point $a\in\mathbb{R}^n$ from $X$ by 
\[d(a,X)=d(a):= \inf\lbrace \|a-x\|\; |\,x\in X\rbrace,\]

which allows us to define the set of closest points of $X$ to $a$ by

\[m_X(a)=m(a):=\lbrace x\in X|\, d(a,X)=\|a-x\|\rbrace.\]

\begin{dfn}\label{definicja}
The main object discussed in this paper is the \textit{medial axis} of a closed set $X$, that is, the set of points of $\mathbb{R}^n$ admitting more than one closest point to the set $X$ i.e.
\[M_X:=\lbrace a\in\mathbb{R}^n|\, \# m(a)>1\rbrace.\]
\end{dfn}
A descriptive way to imagine the medial axis brings a picture of the propagation of a firefront starting at $X$. In this case, the medial axis of $X$ is precisely the set of points where fronts originating from different starting points meet. This picturesque idea illustrates maybe the most profound feature of the medial axis -- it collects exactly those points of the ambient space at which the distance function is not differentiable.

Recall that, for any $x\in X$, we call the limit $\limsup_{t\to 0}\frac{1}{t}(X-x)$ the \textit{(Peano) tangent cone} of $X$ at $x$ and denote it by $C_xX$. Furthermore, if $X$ is definable in an o-minimal structure, the upperlimit in the tangent cone definition turns into the Kuratowski limit due to the curve selection lemma. Finally, the normal cone of $X$ at $x$, denoted by $N_xX$ is defined by 
\[N_xX:=\lbrace x\in \mathbb{R}^n|\,\forall v\in C_aX: \,\langle x,v\rangle\leq 0\rbrace\]
(where $\langle\cdot,\cdot\rangle$ denotes the natural scalar product on $\mathbb{R}^n$). 

There are two subsets of $N_xX$ closely related to the medial axis of special interest to us. Namely the \textit{normal set} of $x\in X$: 
\[\mathcal{N}(x):=\lbrace a\in\mathbb{R}^n|\,x\in m(a)\rbrace\]

and the \textit{univalued normal set} of $x\in X$: 
\[\mathcal{N}'(x):=\lbrace a\in\mathbb{R}^n|\,\lbrace x\rbrace = m(a)\rbrace.\]

As was shown in \cite{BirbrairDenkowski} 

\begin{prop}\label{Properties}
For a definable set $X\subset\mathbb{R}^n$ and $x\in X$:
\begin{enumerate}
    \item The medial axis $M_X$, and the multifunctions $m,\mathcal{N},\mathcal{N}'$ are definable.
    \item $\mathcal{N}'(x)\subset \mathcal{N}(x)\subset N_xX+x$
    \item $M_X=\bigcup_{x\in X} 
\mathcal{N}(x)\backslash\mathcal{N}'(x)$

    \item $\limsup_{X\ni y\rightarrow x}\mathcal{N}(y)\subset \mathcal{N}(x)$. 
\end{enumerate}
\end{prop}

During the development of theory, several ideas for measuring the distance between $X$ and $M_X$ emerged. To name but a few, there is Federer's reach and the reaching radius of Birbrair and Denkowski. We will use a modified version of the reaching radius also found in \cite{Bial1}.

\begin{dfn}
For any point $a\in X$, we write $V_a:=N_aX\cap\mathbb{S}$ to be the set of \textit{directions normal} to $X$ at $a$. Then we denote the \textit{limit set of normal directions} by \[\widetilde{V}_a:=\limsup_{a_\nu\to a}V_{a_\nu}.\]
\end{dfn}

\begin{rem}
If $X$ is a $\mathscr{C}^1$-smooth manifold in a certain neighbourhood of $a\in X$, then, by definition, the tangent spaces and, what follows, the normal spaces are continuous at $a$. Therefore, the limit set of normal directions is just the set of normal directions in such a case.\end{rem}

For a point $a\in X$, we introduce the following definition of a reaching radius.

\begin{dfn}
For $v\in V_a$, we define a \textit{directional reaching radius} by
\[r_v(a):=\sup \lbrace t\geq 0|\,a\in m(a+tv)\rbrace.\]

Then for $v\in \widetilde{V}_a$ we define a \textit{limiting directional reaching radius} by
\[\Tilde{r}_v(a):=\liminf\limits_{X\ni x\to a, V_x\ni v_x\to v\in \widetilde{V}_a}r_{v_x}(x),\]

and finally the \textit{reaching radius} at $a$ is
\[r(a)=\inf_{v\in \widetilde{V}_a}\tilde{r}_v(a).\]
\end{dfn}

The results from \cite{Bial1} and \cite{BirbrairDenkowski} state

\begin{thm}
For any $x\in X$ there is \[r(x)=0 \iff x\in \overline{M_X}\]
\end{thm}

\begin{prop}
The function $\rho: V_a\ni v\rightarrow \rho(v) = r_v(a)\in (0,+\infty]$ is continuous for any $a\in Reg_2 X$. 
\end{prop}

\section{C1-smooth case}

Throughout this section, we assume $X$ to be a $\mathscr{C}^1$-smooth manifold in a neighbourhood of $x\in X$.
Although the medial axis does not detect every $\mathscr{C}^2$-singularity (as one can check by examining a graph of $f(x)=x^2\text{sgn }x$), it is still possible to point out a subset of $ Sng_2 X$ approached by the medial axis. This section is influenced by the characterisation done by Birbrair and Denkowski on the Euclidean plane. According to their findings for planar curves, the superquadracity (see Definition~\ref{superquadraticDfn}) is necessary and sufficient for approaching the $\mathscr{C}^1$-smooth part of the curve by its medial axis. We present here a transition of the theory to higher dimensional spaces. Due to an altered environment, our methods of proof differ from the ones in the mentioned paper.  Unfortunately, the outcome is weaker as the situation grew more complex than the planar one -- an example at the end of the section shows that superquadracity is not enough to detect every instance of medial axes which approaches the underlying set.

\begin{dfn}\label{superquadraticDfn}
Assume that $X$ is a closed definable subset of $\mathbb{R}^n$. We say that $X$ is \textit{superquadratic} at $a\in Sng_2X\cap Reg_1X$ if a function \[g(\varepsilon):=\max_{x\in X,\|p(x)-a\|=\varepsilon}\|x-p(x)\|,\] where $p(x)$ is an orthogonal projection on $T_aX+a$, has a limit 
\[\lim_{\varepsilon\to 0^+}g(\varepsilon)/ \varepsilon^2\to+\infty.\]

We collect all points of $X$'s superquadracity in a set
\[\mathcal{SQ}(X):=\lbrace x\in X| \; X \text{ is superquadractic at }x\rbrace.\]

\end{dfn}
\begin{rem}
\begin{enumerate}
    \item Clearly, for a definable set $X\subset\mathbb{R}^n$ and any point $x\in\mathcal{SQ}(X)$, the local dimension of $X$ at $x$, that is \[\dim_x X:=\min\{\dim (X\cap U)\mid U\text{-neighbourhood of }x\},\] must satisfy $0<\dim_xX<n$.
    \item If $X$ is a set definable in an o-minimal structure which is polynomially bounded, function $g(\varepsilon)$ can be written in a form \[g(\varepsilon)=\alpha\varepsilon ^{\eta} +o(\varepsilon^\eta)\text{ with }\alpha,\eta\in \mathbb{R},\alpha> 0.\] In such a case, the exponent $\eta$ is equal to what is known as the order of $g$ at zero:
    \[ord_0 g: = \sup\{\theta > 0 \mid |g (x)| \leq const.\|x\|^\theta , \quad \|x\| \ll 1\}.\]
    Then, $X$ is superquadratic at one of its points if and only if $\eta\in (1,2)$. Furthermore, in the case of hypersurfaces definable in polynomially bounded structures, the definition simplifies to the one proposed in \cite{BirbrairDenkowski}.
    \item Two cornerstone examples of functions bearing superquadratic graphs at $0$ are $f(x)=|x|^{3/2}$ and $f(x)=x^2\ln|x|$ extended through zero (the latter one not belonging to any polynomially bounded structure). The medial axes of their graphs approach the origin, making them inseparable from the graphs.
\end{enumerate}
\end{rem}

Being interested in conditions under which $\overline{M_X}\cap X$ is not void, we can rotate, translate and even scale $X$ freely. Therefore we restrict the study of superquadratic points to a situation where $a=0$ and $T_aX=\mathbb{R}^k\times\lbrace 0\rbrace^{n-k}$ ($k=\dim_aX$). In that case, we identify $\mathbb{R}^n$ with the Cartesian product $\mathbb{R}^k\times\mathbb{R}^{n-k}$ with coordinates $x$ and $y$ (formula for $g(\varepsilon)$ is written respectively as $\max_{(x,y)\in X,\|x\|=\varepsilon}\|y\|$). It is easy to observe then (for example, by using the definable cell decomposition adapted to $X$) that a definable set $X$ yields a definable function $g$. 


We begin our study of superquadracity by seizing control over the codimension of the set. For that purpose, we prove a result about the projections of superquadratic sets.

\begin{prop}\label{Superqadratic projection}
Let $X$ be a definable set in $\mathbb{R}^n$ of dimension $k$ with $0\in Reg_1 X$ and $T_0X=\mathbb{R}^k\times\lbrace 0\rbrace^{n-k}$. Denote by $\pi_i:\mathbb{R}^n\rightarrow\mathbb{R}^{k+1}$ a natural projection onto the first $k$ and $(k+i)$th coordinates. Then $0\in\mathcal{SQ}(X)$ if and only if there exists $j\in \lbrace 1,\ldots,n-k\rbrace$ such that $0\in\mathcal{SQ}(\pi_j(X)).$
\end{prop}
\begin{proof}
The sufficiency of the condition is self-evident; it is enough to remark that  $g_j(\varepsilon)\leq g(\varepsilon)$ holds for every $\varepsilon\geq 0.$

Let us prove the necessity of the condition. Assume $X$ to be superquadratic at the origin. Inferring from the position of the tangent cone of $X$, each of the projections $\pi_i(X)$ forms a $\mathcal{C}^1$-submanifold in a neighbourhood of the origin, so we only need to estimate the growth of $y_i$-the $(k+i)$-th coordinate of points in $X$.
We can write
\[\lim_{\varepsilon\to 0^+}g(\varepsilon)/\varepsilon^2=\lim_{\varepsilon\to 0^+}\left(\max_{(x,y)\in X,\|x\|=\varepsilon}\|y\|\right)/\varepsilon^2=+\infty\] 
Now, for every $i$, a projection $\pi_i(X)$ is a definable set, and \[\max_{(x,y)\in X,\|x\|=\varepsilon}|y_i|=\max_{(x,y_i)\in \pi_{i}(X),\|x\|=\varepsilon}|y_i|.\] 
At the same time, \[\max_{(x,y)\in X,\|x\|=\varepsilon}\|y\|\leq \max_{(x,y)\in X,\|x\|=\varepsilon}|y_1|+\ldots+|y_{n-k}|\leq \] \[\leq\max_{(x,y)\in X,\|x\|=\varepsilon}|y_1|+\ldots+\max_{(x,y)\in X,\|x\|=\varepsilon}|y_{n-k}|.\]
Moreover, each $g_i(\varepsilon):=\max_{(x,y)\in X,\|x\|=\varepsilon}|y_i|$ is a definable function defined in a neighbourhood of the origin. Thus, for a neighbourhood small enough, there exists $j$ such that $g_j\geq g_i$. We can estimate

\[\max_{(x,y)\in X,\|x\|=\varepsilon}|y_1|+\ldots+\max_{(x,y)\in X,\|x\|=\varepsilon}|y_{n-k}|\leq(n-k)\max_{(x,y)\in X,\|x\|=\varepsilon}|y_j|. \]

In other words, we have 
\[g(\varepsilon)/\varepsilon^{2}\leq (n-k) g_j(\varepsilon)/\varepsilon^{2}.\]
Since $g(\varepsilon)/\varepsilon^{2}$ tends to infinity as $\varepsilon\to 0$, the function $g_j(\varepsilon)/\varepsilon^{2}$ has to behave alike. 
\end{proof}

\begin{rem}
Observe that the last inequality of Proposition~\ref{Superqadratic projection} can be transformed into 
\[g(\varepsilon)\varepsilon^{-\eta}\leq (n-k) g_j(\varepsilon)\varepsilon^{-\eta}\leq (n-k)g(\varepsilon)\varepsilon^{-\eta}\]
and, since $g_j(\varepsilon)\varepsilon^{-\eta}$ is a definable function, it has a limit as $\varepsilon\to 0^+$. If $g(\varepsilon)=\alpha\varepsilon^\eta+o(\varepsilon^\eta)$ with $\alpha>0$, then the limit $\lim_{\varepsilon\to 0^+} g_j(\varepsilon)\varepsilon^{-\eta}$ has to be positive and finite due to estimations above. Hence $g_j(\varepsilon)$ has to be of the form $\tilde{\alpha}\varepsilon^{\eta}+o(\varepsilon^{\eta})$ for a certain $\tilde{\alpha}>0$ as well.
\end{rem}

\begin{cor}
Let $X$ be superquadratic at $0$ and let $\dim_0 X=k$, then there exists a neighbourhood $U$ of the origin and $(k+1)$-dimensional vector subspace $L$ containing $T_0X$ such that an orthogonal projection of $X\cap U$ on $L$ is superquadratic at the origin. 
\end{cor}

Proposition~\ref{Superqadratic projection} allows studying the superquadracity in a purely geometric manner thanks to the next result.

\begin{prop}\label{Superquadratic interestion with balls}
Let $X$ be a definable subset of $\mathbb{R}^n$. Assume that $0\in\mathcal{SQ}(X)$, then there exists a vector $v\in V_0$ such that for all $r>0$, the intersection $\mathbb{B}(rv,r)\cap X$ is not empty.
\end{prop}
\begin{proof}
The proposition is obvious for $\text{codim } X=1$ since the lower hemisphere of $r\mathbb{S}^{n-1}$ is not superquadratic for any $r>0$. Indeed, take \[\varphi:\mathbb{B}(0,r)\ni x\to (x,r-\sqrt{r^2-\|x\|^2})\in\mathbb{R}^{n},\] then $\lim_{x\to 0}\varphi_{n}(x)/\|x\|^2= 1/(2r)<\infty$. It is plain to see that some points of superquadratic hypersurface $X$ with $T_0X=\mathbb{R}^{n-1}\times{0}$ has to lie in one of the balls $\mathbb{B}((0,\pm rv),r)$ for all $r>0,v\in V_0$.

Assume now that $\text{codim } X>1$. Without loss of generality we can take $T_0X=\mathbb{R}^k\times\lbrace 0\rbrace^{n-k}$, then there must be $\sum_{i=1}^k x_{i}^{2}>\sum_{i=1}^{n-k} y_{i}^2$ for points in a certain neighbourhood of the origin.
Since the notion of superquadracity is local, we can assume that the inequality holds for all points of $X$.

While $\text{codim } X>1$, there exists $i\in\lbrace1,\ldots,\text{codim } X\rbrace$ such that $\pi_i(X)$ is superquadratic at the origin due to Proposition~\ref{Superqadratic projection}. 
We will show that for $v=e_{k+i}$ - the $(k+i)$th vector from the canonical basis of $\mathbb{R}^n$, the assertion holds. Take $r>0$; we need to prove that there exists a point $\xi\in\mathbb{B}(rv,r)\cap X$. 

Since $\pi_i(X)$ is superquadratic, there exists $\tilde{v}\in V_0\pi_i(X)=\lbrace 0\rbrace^{k}\times\{-1,1\}$ such that for all $\tilde{r}>0$, the intersection $\mathbb{B}(\tilde{r}\tilde{v},\tilde{r})\cap \pi_i(X)$ is not empty. Reflecting $X$, if necessary, we can assume $\tilde{v}=(0,\ldots,0,1)$.
By taking $\tilde{r}=\frac{1}{2}r$, we get the nonemptiness of $\mathbb{B}(\frac{1}{2}r\tilde{v},\frac{1}{2}r)\cap \pi_{i}(X)$. In other words, we have found such $(x_1,\ldots,x_{k},y_{i})\in \pi_i(X)$ that an inequality~$\sum_{j=1}^k x_j^2+(\frac{1}{2}r-y_{i})^2<\frac{1}{4}r^2$ holds.

For any $(x,y)\in \pi_i^{-1}((x_1,\ldots,x_{k},y_{i}))\cap X$, we compute
\[\|(x,y)-rv\|^2=\sum_{j=1}^{k} x_{j}^{2} + \sum_{j=1,j\neq i}^{n-k} y_{j}^2 +(y_{i}-r)^2<\]
\[<2\sum_{j=1}^{k} x_{j}^{2}+(y_{i}-\frac{1}{2}r-\frac{1}{2}r)^2=2\sum_{j=1}^{k} x_{j}^{2}+(y_{i}-\frac{1}{2}r)^2-y_{i} r+\frac{3}{4}r^2=\]

\[=2\sum_{j=1}^{k} x_{j}^{2}+2(y_{i}-\frac{1}{2}r)^2-(y_{i}-\frac{1}{2}r)^2-y_{i}r+\frac{3}{4}r^2<\]
\[<\frac{1}{2}r^2-y_{i}^2+y_{i}r-\frac{1}{4}r^2-y_{i}r+\frac{3}{4}r^2<r^2.\]

Thus, $(x,y)\in\mathbb{B}(rv,r)\cap X$ what had to be shown.

\end{proof}

The main result on the set of $\mathscr{C}^2$-singularities of $X$ is the following partial characterisation.

\begin{thm}\label{superquadratic theorem}
Let $X$ be a definable subset of $\mathbb{R}^n$ such that $0\in Sng_2 X\cap Reg_1 X$. Then $0\in\overline{M_X}$ if $0\in\mathcal{SQ}(X)$.
\end{thm}
\begin{proof}
Denote $k:=\dim_0X$. As was mentioned earlier, we can rotate the coordinate system to obtain $T_0X=\mathbb{R}^k\times\lbrace 0\rbrace^{n-k}$. Then, by Proposition~\ref{Superquadratic interestion with balls}, there exists such $v\in V_0$ that for all $r>0$ the intersection $\mathbb{B}(rv,r)\cap X$ is nonempty. In terms of the normal set it means  \[\mathcal{N}(0)\cap \mathbb{R}_+v=\lbrace 0\rbrace.\]

Since $Reg_1 X$ is an open set in $X$, there exists a neighbourhood of the origin $U$ such that  $U\cap X$ forms a $k$-dimensional $\mathscr{C}^1$-submanifold of $\mathbb{R}^n$. Consequently, we have \[\dim T_x X^{\perp} +\dim (T_0(X)+\mathbb{R}v)>n\text{, for all }x\in U\cap X,\] and therefore, for $x\in U\cap X$, a linear space $V(x):=T_x X^{\perp}\cap (T_0(X)+\mathbb{R}v)$ has to be at least one-dimensional.

Assume now that $0\notin \overline{M_X}$ and take $\varepsilon=d(0,\overline{M_X})/2$. Following the equality from Proposition~\ref{Properties}$(3)$ for $x\in \mathbb{B}(0,\varepsilon)\cap U$, an univalued normal set $\mathcal{N}'(x)$ contains a sphere of radius $\varepsilon$ and dimension $(n-k)$. Take then a sequence $x_\nu\in \mathbb{B}(0,\varepsilon)\cap U$ with $x_\nu\rightarrow 0$ and denote by $s_\nu$ an element of $V(x)\cap\varepsilon\mathbb{S}$ for which the scalar product $<s_\nu,v>$ is positive. By choosing a subsequence, we can assume now that \[s_\nu\to s\in \varepsilon\mathbb{S}\text{ and }\mathcal{N}'(x_\nu)\ni x_\nu+s_\nu\rightarrow 0+ s.\]
Hence, $s\in \limsup_{X\ni x\rightarrow 0}\mathcal{N}'(x)$. Since the upper limit of $\mathcal{N}'(x)$ is a subset of $\mathcal{N}(0)$ by proposition~\ref{Properties} we have $s\in \mathcal{N}(0)$.

Moreover, since the origin is a point of the $\mathscr{C}^1$-regularity of $X$, the normal spaces $T_{x_\nu}X^\perp$ converge to the normal space at the origin. Consequently, the choice of $s_\nu$ guarantees that $s\in \mathbb{R}v$.

Therefore $s$ has to be a point of $\mathcal{N}(0)\cap\mathbb{R}v$. Since $s$ has a nonnegative scalar product with $v$ and a positive norm, we conclude that $\mathcal{N}(0)\cap \mathbb{R}_+v$, which ends in a contradiction.
\end{proof}

Theorem~\ref{superquadratic theorem} can be inverted if $X$ is a curve.

\begin{thm}\label{Superqadratic curve}
Under the assumptions of the previous theorem, if $X$ is a definable curve 
and $0\in\overline{M_X}\cap Reg_1X$, then $0\in \mathcal{SQ}(X)$. 
\end{thm}

\begin{proof}
Since $0\in Reg_1 X$, we can describe $X$ in a certain neighbourhood $U$ of the origin as a graph of a definable function \[F:(-\varepsilon,\varepsilon)\ni x\rightarrow (f_1(x),\ldots,f_{n-1}(x))\in\mathbb{R}^{n-1}\] smooth on $(-\varepsilon,\varepsilon)\backslash\{0\}.$ Divide $F$ into branches $F_-:=F|(-\varepsilon,0)$ and $F_+:=F|(0,\varepsilon)$, and denote their components by $f_{i,-},f_{i,+}$, respectively. If $X$ is not superquadratic at the origin, then, according to Proposition~\ref{Superqadratic projection}, neither any of $f_{i,\pm}$ can be. That means every $f_{i,\pm}$ has a finite limit $\lim_{x\to 0}f_{i,\pm}(x)/x^2$ due to the structure's o-minimality. 
Therefore every $f_{i,\pm}$ admits a $\mathscr{C}^2$-extension through $0$. Accordingly, both $F_-$ and $F_+$ can be extended to $\mathscr{C}^2$-functions through the origin.  Now, since $\mathscr{C}^2$ submanifolds are disjoint with their medial axes, we can find fixed radii $r_-,r_+$ such that a ball of radius $r_\pm$  centred at $x+vr_\pm$ is disjoint with the graph $\Gamma_{F_\pm}$ for any $x\in \Gamma_{F_\pm}, \, v\in V_x$. The smaller radius holds the property for the whole $X\cap U$, thus $X\cap\overline{M_X}=\emptyset$.
\end{proof}

Naturally, $\overline{M_X}$ is a closed set; thus, an easy corollary follows.

\begin{cor}\label{closure of superquadratic}
For a closed definable set $X$, there is \[\overline{\mathcal{SQ}(X)}\subset \overline{M_X}\cap Reg_1X.\] Moreover, if $\dim X=1$ 
, the set $\mathcal{SQ}(X)$ is discrete, and the inclusion becomes equality.
\end{cor}
\begin{proof}
The first part is immediate from Theorem~\ref{superquadratic theorem}. The second follows Theorem~\ref{Superqadratic curve} and the fact $\mathcal{SQ}(X)\subset Sng_2X$ the latter set being discrete.
\end{proof}

Corollary \ref{closure of superquadratic} gives full information about $\overline{M_X}\cap Reg_1 X$ when $X$ is a subset of a plane. The problem becomes more complicated in higher dimensions, and Theorem~\ref{superquadratic theorem} cannot be reversed.

\begin{ex}\label{kontrsuperq}
First, observe that for a cone $C=\lbrace z^2=x^2+y^2\rbrace$ and a point $p$ with $x,y$ coordinates equal to zero, $m(p)$ is a full circle parallel to the XY plane, and $\overline{ M_C\cap\lbrace x=y=0\rbrace}\cap C=\lbrace 0\rbrace$.
Next, let us consider $X$ - a graph of a function 
\[ f(x,y)=\begin{cases}
        \frac{y^2}{x}, &\quad |y|<x^3, x>0\\
        2x^2|y|-x^5, &\quad |y|\geq x^3, x>0\\
        0, &\quad x\leq 0\\
\end{cases}.
\]
One can check that $f$ is $\mathscr{C}^1$ smooth, and its graph is a semialgebraic set not superquadratic at any point of $\mathbb{R}^3$. A part of the graph above the region $|y|<x^3$ is a part of the cone $C$ rotated in such a manner that the $x$-axis belongs to it. Every point $(x,0)$ belongs to $|y|<x^3$ with some neighbourhood. Therefore neighbourhoods of points $(x,0,0)$ in the graph are parts of the rotated cone $C$. Consequently, for any $x>0$, we have $(x,0,x)=\mathcal{N}((x,0,0))\backslash \mathcal{N}'((x,0,0))\in M_X$. Taking any sequence $x_\nu\rightarrow 0$, we obtain $0\in \overline{M_X}\cap X$.
\end{ex}
\begin{figure}
    \centering
    \includegraphics[width=0.8\textwidth]{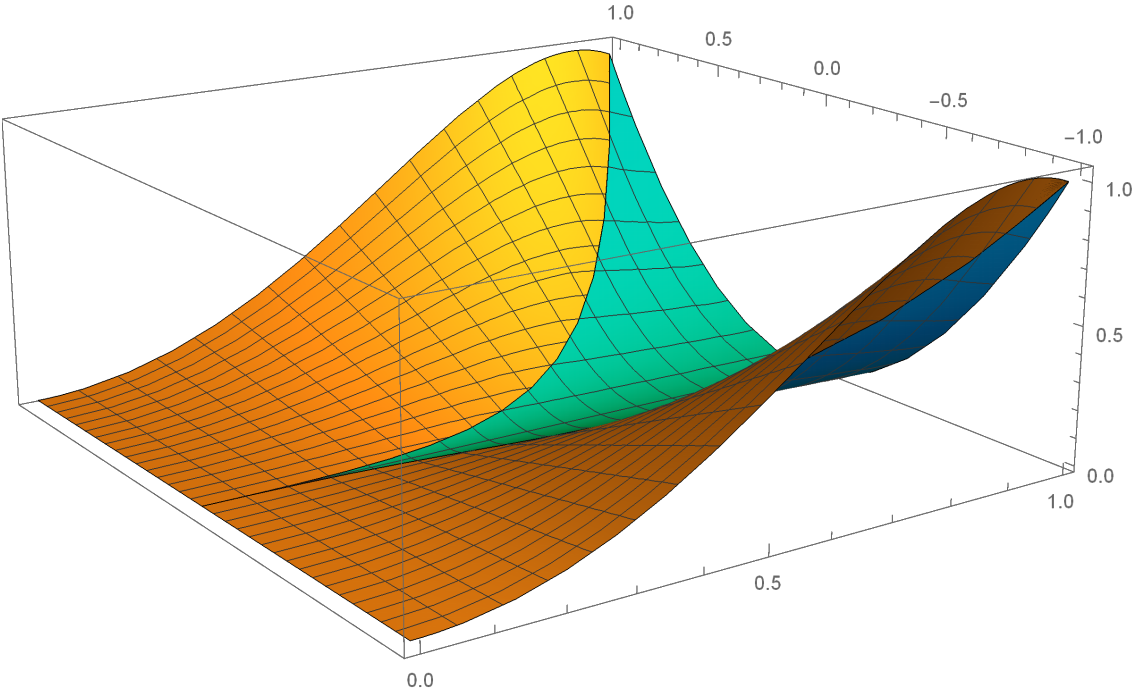}
    \caption{The graph from Example~\ref{kontrsuperq}}
    \label{fig:counter}
\end{figure}

\section{C1-singular case}

Findings of Rataj and Zaj{\'i}cek \cite{Rataj} state that every closed topological manifold with a positive reach has to be a $\mathscr{C}^{1,1}$-manifold (that is, a manifold admitting a parametrisation with Lipschitz derivative). Conversely, we can say that every point of a closed topological manifold in which the differentiability class is lower than $\mathscr{C}^1$ has to be approached by the manifolds' medial axis. Though, the regularity class postulated above is by no means necessary as we could see in the previous section. This section answers the extent to which we can loosen the assumption of the set being a manifold.

At start, it is worth noting that Motzkin Theorem~\cite{Motzkin} combined with the result from \cite{BirbrairDenkowski} (stating that $M_{C_0X}\subset C_0M_X$ for any definable $X\ni 0$) gives an appealing and fast 
\begin{thm}{(Tangent cone criterion)}\label{tangent cone criterion}
Let $X$ be definable. Take a point $x\in X$ and assume that $C_xX$ is nonconvex. Then $x\in \overline{M_X}$. 
\end{thm}
\begin{proof} 
We can assume $x=0$. Since $C_0X$ is nonconvex, Motzkin Theorem allows finding a point $p\in M_{C_0X}$. Since the medial axis commutes with homotheties, \[\forall t>0:\quad t p\in tM_{C_0X}=M_{tC_0X}=M_{C_0X},\] we have $0\in \overline{M_{C_0X}}$. The inclusion $M_{C_0X}\subset C_0M_X$ finishes the proof.
\end{proof}
Sadly not every point of $Sng_1X$ induces a nonconvex tangent cone (consider, for example, $0\in X:=(-\infty,0]=C_0X$). More than that, the problem remains even if we assume that $C_x X$ is flat (cf. Example~\ref{Ghomi}).

Being aware of high dimensional problems, it is easier now to appreciate the Euclidean plane case tameness. Definable structures assure that a planar set can be decomposed into a finite family of curves and areas between them in the neighbourhood of every point around which it is not $\mathscr{C}^1$-smooth. Using this decomposition, Birbrair and Denkowski presented in \cite{BirbrairDenkowski} a complete metric characterisation of the planar medial axes. One of the possible approaches to subsets of higher dimensional spaces starts with a description of points of a set's $\mathscr{C}^1$-smoothness by their paratangent cones given by F. Bigolin and G. Greco \cite{BigolinGrecko}. Their results give us a slightly more refined tool to analyse the behaviour of a medial axis than the tangent cone one.

\begin{thm}\label{lepszy motzkin}
Let $X$ be a definable closed subset of $\mathbb{R}^k$. If $x_0\notin \overline{M_X}$, then \[C_{x_0}X\subset \liminf_{X\ni x\to x_0} C_{x}X.\]

\begin{proof}
By the tangent cone criterion, whenever $x_0\notin \overline{M_X}$, a neighbourhood of $x_0$ exists where all the tangent cones of $X$ are convex.
Moreover, $x_0\notin \overline{M_X}$ if and only if the reaching radius is separated from zero in some neighbourhood $U$ of $x_0$. Otherwise, $r(x_0)=0$ and consequently, $x_0$ is a point of the medial axis closure. 

Since the directional reaching radius $r_v(x)$ is positive for any direction in $V_x$, the space normal to $X$ at a point $x$ is a cone spanned over the suitably translated normal set $\mathcal{N}(x)$. \[N_xX=\mathbb{R}_{\geq 0}(\mathcal{N}(x)-x).\] 

Recall that $\limsup_{x\to x_0}\mathcal{N}(x)\subset \mathcal{N}(x_0)$. This convergence will not be affected by moving the sets $\mathcal{N}(x)$ by a vector $x$. Furthermore, spanning a cone on a set is inclusion preserving, thus \[\mathbb{R}_{\geq 0}\limsup_{X\ni x\to x_0}(\mathcal{N}(x)-x)\subset \mathbb{R}_{\geq 0}(\mathcal{N}(x_0)-x_0).\] Now, since none of the reaching radii $r_v(x)$ converge to zero, we can move the spanning operation inside the upper limit. Therefore, with the established equivalence between the normal space and the normal set, we can write that
\[\limsup_{X\ni x\to x_0} N_xX\subset N_{x_0}X.\] Next, by taking the normal cone to both sides, we invert the inclusion, obtaining 
\[N_0\limsup_{X\ni x\to x_0} N_xX\supset N_0N_{x_0}X.\]

We will prove now the inclusion \[\liminf_{x\to x_0}N_0N_xX\supset N_0\limsup_{x\to x_0} N_xX.\]
Take a vector $v\in N_0\limsup_{x\to x_0} N_xX$. Since the upper limit of cones is a cone, by definition,  $v$ has to form a nonpositive scalar product with every vector in $\limsup_{x\to x_0} N_xX$. It means that for any sequence of points $v_x\in N_{x} X$, the upper limit $\limsup_{x\to x_0} \langle v,v_x/\|v_x\|\rangle $ is nonpositive. 

Note that, since the normal cone is indeed a cone, we have \[N_0N_xX=\bigcap_{v_x\in N_xX}\lbrace w\in\mathbb{R}^n|\,\langle w,v_x\rangle \leq 0\rbrace.\]
Now, should $v\notin \liminf_{x\to x_0}N_0N_xX$, a sequence $x_\nu\to x_0$ would exists with $N_0N_{x_\nu}X$ separated from $v$. However, it means that for a certain sequence of points $x_\nu\in X$ and normal vectors $v_\nu\in N_{x_\nu}X$, the scalar product $\langle v,v_\nu/\|v_\nu\|\rangle $ stays positive and separated from zero.

Since $N_xX$ and $C_xX$ are convex, $N_0N_xX=C_xX$, and the assertion is just an alternate form of the inclusion
\[\liminf_{X\ni x\to x_0}C_xX=\liminf_{X\ni x\to x_0}N_0N_xX\supset N_0N_{x_0}X=C_{x_0}X.\]

\end{proof}

\end{thm} 

Note here that the theorem above generalises the tangent cone criterion in the o-minimal setting. 
\begin{prop}
Let $X$ be a definable closed subset of $\mathbb{R}^n$, and let $x_0$ be a point of $X$. If $C_{x_0}X$ is nonconvex, then \[C_{x_0}X\not\subset \liminf_{X\ni x\to x_0} C_xX.\]
\end{prop}
\begin{proof}
Assume that $x_0=0$ and the tangent cone $C_0X$ is nonconvex. Then, thanks to the tangent cone criterion, \[0\in\overline{M_{C_0X}}\subset C_0 M_X.\]

Fix now a vector $v\in M_{C_0X}$ and find a curve $\psi:[0,1]\to M_X$ tangent to $v$, meaning $C_0\psi([0,1])=\mathbb{R}_{\geq 0}v$. Possibly after a reparametrisation of $\psi$, we can choose a curve $\gamma:[0,1]\ni t\to \gamma(t)\in m_X(\psi(t))$ tangent to a certain $s\in m_{C_0X}(v)$. Note that the change of $\psi$ parametrisation is needed only to ensure the continuity of $\gamma$. To justify the $s\in m_{C_0X}(v)$ part of $\gamma$ definition observe that $\lim_{t\to 0}\gamma(t)/\|\gamma(t)\|\in C_0X$. Thus, its direction cannot be closer to $v$ than any element of $m_{C_0X}(v)$. On the other hand, should there exist a vector $s'\in C_0X$ closer to $v$ than $s$, $\gamma(t)$ would not realise distance for $\psi(t)$ for small $t>0$.
Now, since $v$ was a point of the tangent cone's medial axis, a point $r\in m_{C_0X}(v)\backslash \lbrace s\rbrace$ exists. For such $r$, we have \[\langle v-s,r\rangle =\langle v-r+r-s,r\rangle =\langle r-s,r\rangle =\|r\|^2-\langle s,r\rangle >0\]
since $\|r\|=\|s\|$ due to the conic structure of $C_0X$ and since the vectors $r$ and $s$ are not collinear.
Moreover, $\psi(t)-\gamma(t)\in N_{\gamma(t)}X;$ thus, \[C_{\gamma(t)}X\subset\lbrace w\in\mathbb{R}^n|\,\langle\psi(t)-\gamma(t),w\rangle \leq 0\rbrace.\]

Consequently,
\[\liminf_{X\ni x\to 0}C_xX\subset \liminf_{t\to 0} \lbrace w\in\mathbb{R}^n|\,\langle\frac{\psi(t)-\gamma(t)}{\|\gamma(t)\|},w\rangle <0\rbrace.\]

Next, since $\gamma$ and $\psi$ are tangent to $s$ and $v$ respectively, it is possible to calculate $\lim_{t\to 0} \frac{\|\psi(t)\|}{\|\gamma(t)\|}.$ Indeed by taking $t_0$ close to zero we can demand $\psi((0,t_0)),\gamma((0,t_0))$ to be subsets of the cones \[V(r):=\mathbb{R}_+\cdot\mathbb{B}(v,r),\;S(r):=\mathbb{R}_+\cdot\mathbb{B}(s,r)\] respectively, with $r$ arbitrary small. Consider the plane $P:=span\lbrace v,s\rbrace$ and denote respectively by $C(r),F(r)$ the rays in $P\cap S(r)$ closest and furthest to $v$.  Having fixed $t_1\in (0,t_0)$ we can see that $\|\gamma(t_1)\|\in [l(r),u(r)]$ where \[l(r):=\min\lbrace p(x)|x\in V(r),\|x\|=\|\psi(t_1)\|\rbrace,\; p\text{ is the projection on }F(r)\] and \[u(r):=\max\lbrace q(x)|x\in V(r),\|x\|=\|\psi(t_1)\|\rbrace,\; q\text{ is the projection on }C(r).\] By taking the limit $r\to 0$ we observe that \[\lim_{t\to 0} \frac{\|\psi(t)\|}{\|\gamma(t)\|}=\frac{\|v\|}{\|s\|}.\]

Therefore, we compute \[\lim_{t\to 0} \langle\frac{\psi(t)-\gamma(t)}{\|\gamma(t)\|}, r\rangle =\lim_{t\to 0} \langle\frac{\psi(t)}{\|\psi(t)\|}\frac{\|\psi(t)\|}{\|\gamma(t)\|}-\frac{\gamma(t)}{\|\gamma(t)\|}, r\rangle =\langle\frac{v}{\|v\|}\frac{\|v\|}{\|s\|}-\frac{s}{\|s\|},r\rangle .\]

The limit is positive; hence the approximating scalar products have to become positive for $t$ close to zero. In the end,
\[r\notin \liminf_{X\ni x\to 0} C_{x}X.\]

\end{proof}

The new method is strong enough to settle cases unreachable for the tangent cone criterion, such as a beautiful surface from \cite{Ghomi}.

\begin{figure}
    \centering
    \includegraphics[width=0.7\textwidth]{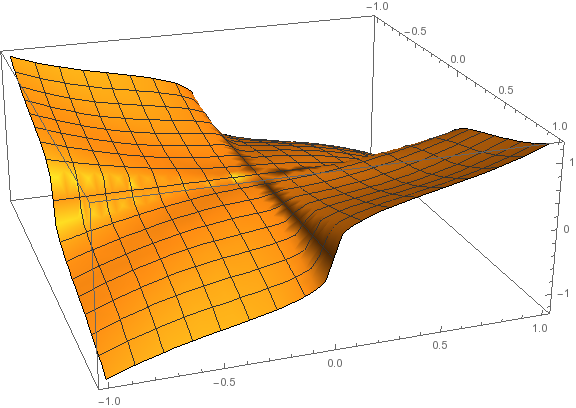}
    \caption{The surface from the Example~\ref{Ghomi}.}
    \label{fig:Ghomi}
\end{figure}

\begin{ex}(Howard, Ghomi)\label{Ghomi}
Let $X:=\lbrace z^3=xy(x^4+y^4)\rbrace\subset\mathbb{R}^3$ (see Figure~\ref{fig:Ghomi}). It is a surface, and the tangent cone to $X$ at any point is a plane. However, it is not $\mathscr{C}^1$-smooth at the origin. The tangent cones at points lying on the $y$-axis are constantly equal to $\lbrace 0 \rbrace\times\mathbb{R}^2$, whereas the tangent cone at the origin is equal to $\mathbb{R}^2\times\lbrace 0 \rbrace$. Therefore, the lower limit of tangent cones cannot be a superset of the tangent cone at the origin, and consequently, the medial axis of $X$ reaches the~origin.
\end{ex}
 
Sadly, an analogous inequality involving the tangent cone and the upper Kuratowski limit does not ensure an approach of the medial axis. One can see it easily from the example of a closed ball $B\subset\mathbb{R}^n$. For any boundary point, a tangent cone to $B$ is a half-space, whereas the upper limit of the cones comes as the whole space.

In case the studied set is a low-dimensional topological manifold, it is possible to spot the approaching medial axis just by observing an unusual tangent cone at a point. 

\begin{prop}
Let $\Gamma\subset\mathbb{R}^n$ be a $k$-dimensional definable topological submanifold with $k<3$. Assume that $x\in \Gamma$ and $C_{x}\Gamma$ is convex and does not form a $k$-dimensional vector space. Then there exists a sequence $Reg_2\Gamma\ni x_\nu\to x$ with $r(x_\nu)\to 0$ and, consequently, $x\in\overline{M_X}$. 
\end{prop}
\begin{proof} Without loss of generality, we can assume that $x=0$.

If $\Gamma$ is a definable curve, then the tangent cone at $0$ must be spanned over a single nonzero vector $v$. Denote by $\Gamma_+$ one of the two branches of $\Gamma$ starting at the origin. If all the reaching radii $r(x_\nu)$ were greater than a positive real number $R$ for $x_\nu$ in some neighbourhood of $0$, then for a certain $\varepsilon>0$ all points in \[\lbrace (u,w)\in span(v)^\perp\times span(v)=\mathbb{R}^n|\; \|u\|<R/2,\, \|w\|<\varepsilon \rbrace  \] would have a unique closest point in $\Gamma_+$, leaving no room for the other branch of $\Gamma$. 

For the rest of the proof, assume that $\dim \Gamma=2$. Since $C_0\Gamma$ is convex and at most two-dimensional, it can contain at most two linearly independent vectors. By~rotation, we can assume, therefore, that $C_0\Gamma$ is a subset of $\mathbb{R}^2\times\lbrace 0\rbrace ^{n-2}$.
Let $h$ be a definable homeomorphism between $\mathbb{R}^2$ and a neighbourhood of the origin in $\Gamma$ such that $h(0)=0$.

Assume that the assertion does not hold; that is, \[\exists R>0, U\text{-a neighbourhood of }x,\forall x\in Reg_2\Gamma\cap U:\; r(x)>R.\] Then,
$T:=\lbrace (x,y)\in\mathbb{R}^2\times \mathbb{R}^{n-2}|\,(\|y\|-R)^2+\|x\|^2\leq R^2\rbrace$ meets $\Gamma$ solely at the origin. Additionally, for any curve $\gamma:[0,1]\to\Gamma\cap U$ of the $\mathscr{C}^2$ class with an image contained in $Reg_2\Gamma$ for which $\lim_{t\to 0}\gamma(t)=0$, there is $r(\gamma(t))>R$. Since $Reg_2\Gamma$ is open and dense in $\Gamma$, we can choose $\gamma$ tangential to any vector from $C_0\Gamma$. Since $\Gamma$ is a submanifold, the tangent cone $C_0\Gamma$ cannot form a vector space of dimension one. Therefore there exists a vector $v\in Reg_1 \,C_0\Gamma$ such that $-v\notin C_0\Gamma$. Let the curve $\gamma$ be tangential at the origin to that direction.

Note now that $h^{-1}\circ\gamma$ is a definable curve in $\mathbb{R}^2$ starting at the origin. Therefore by shrinking $\gamma$ we can ensure that $h^{-1}\circ\gamma([0,1])$ does not disconnect $\mathbb{R}^2$. Moreover, it is possible to form a set $G$ by extending $\gamma(t)$ in the directions normal to $\Gamma$ up to the distance of $R$ without ever touching $\Gamma$. Then, $-G$ also remains disjoint with $\Gamma$.

To observe the final contradiction shrink $U$, making it disconnected by the union $Z:=G\cup -G\cup T.$ Now, both parts of $U\backslash Z$ have a piece of $\Gamma$ in them; what follows is that $Z$ disconnects $\Gamma$ into two separated parts. However, $\Gamma\backslash Z$ is the image under the homeomorphism $h$ of the connected set $\mathbb{R}^2\backslash h^{-1}(\gamma([0,1]))$.

\end{proof}

This remark concludes the investigation of the tangent cone and the medial axis. The following results concentrate on the assumption of the set being a topological manifold. To that end, we will call a point $x\in X$ \textit{non-marginal} whenever a $(\dim_xX)$-dimensional topological submanifold $x\in\Gamma\subset X$ can be found in a neighbourhood of $x$ in $X$ and \textit{marginal} in the opposite case. It is easy to check that non-marginal points are dense in $X$. They can be thought of as the `main' part of the set.

\begin{prop}\label{osiaganie zera na reg2}
Let $\Gamma\subset\mathbb{R}^n$ be a definable topological submanifold and take $x_0\in \Gamma\cap \overline{M_\Gamma}$. Then there exists a sequence $Reg_2\Gamma\ni x_\nu\to x_0$ with $r(x_\nu)\to 0$. 
\end{prop}
\begin{proof}
Assume the assertion does not hold. Then there exists a neighbourhood $V$ of $x_0$ and $R\geq 0$ such that the reaching radius $r(x)$ satisfies \[r(x)\geq R\text{ for all } x\in Reg_2\Gamma\cap V.\]

Take a point $a\in Reg_2(Sng_2\Gamma)\cap V$. Then the reaching radius calculated for the set $Sng_2\Gamma$, denoted by $r_{Sng_2\Gamma}(a),$ is greater than $\varepsilon>0$, and consequently, since the reaching radius is continuous on $\mathscr{C}^2$-smooth part of any set (cf. \cite{Bial1}), there exists a neighbourhood $U$ of $a$ such that $r_{Sng_2\Gamma}|U>\varepsilon$.

Now, should $a\in \overline{M_\Gamma}$, there would exist a sequence of points $a_\nu\in M_\Gamma$ convergent to $a$. What would follow is that any sequence of points $b_\nu\in m(a_\nu)$ would also converge to $a$. But that is a contradiction, as for $\nu$ large enough it would mean $b_\nu\in (Reg_2\Gamma\cup(Reg_2(Sng_2\Gamma)\cap U))\cap V$ and $r(b_\nu)\leq \|b_\nu-a_\nu\|\to 0$.

Thus, for every $a\in Reg_2(Sng_2\Gamma)\cap V$ the reaching radius $r(a)$ is positive. Now results of \cite{Rataj} give us the continuity of tangent and normal spaces at $a$ since $\Gamma$ is a topological manifold and $a$ is separated from the medial axis. Since $Reg_2\Gamma$ is dense, that means we can approach any $v\in V_a$ by directions $v_\nu\in V_{x_\nu}$ with $Reg_2\Gamma\ni x_\nu\to a$. For such a sequence, we have $x_\nu\in m(x_\nu+Rv_\nu)$. After passing with $\nu$ to infinity, we obtain 
\[m(a+Rv)=m(\lim_{\nu\to\infty} x_\nu +Rv_\nu)\supset\limsup_{\nu\to\infty} m(x_\nu+Rv_\nu)\supset\lim_{\nu\to\infty} \lbrace x_\nu\rbrace=\lbrace a\rbrace;\]
thus $r(a)\geq R$ for any $a\in Reg_2(Sng_2\Gamma)\cap V$.

For any definable set $X$, the singular part $Sng_2 X$ is a definable subset of $X$ of a dimension lower than $\dim X$. Therefore after at most $\dim \Gamma$ repetitions of the reasoning above applied to regular and singular parts of the $Sng_2(\ldots(Sng_2\Gamma))$, we conclude that no point of $\Gamma\cap V$ can belong to the closure of $M_\Gamma$. It is a clear contradiction with the assumptions.
\end{proof}

\begin{rem}
Observe that the assumption about the topological structure of $\Gamma$ was needed for obtaining $r(a)\geq R$ for any $a\in Reg_2(Sng_2\Gamma)$. Such an intermediate result is crucial for the dimension reduction step $Sng_2\Gamma\to Sng_2(Sng_2\Gamma)$. However, we do not need it if the latter set is empty. In explicit\'e, the assertion holds for any set $\Gamma$ with empty$Sng_2(Sng_2\Gamma)=\emptyset$. It means, in particular, it also holds for topological manifolds with a boundary of $\mathscr{C}^2$ class.
\end{rem}

We will use the proposition to prove that glueing new pieces to a manifold approached by a medial axis cannot separate the medial axis from the obtained union.

\begin{thm}\label{powiekszanie zbioru osiąganego przez szkielet}
Let $\Gamma\subset\mathbb{R}^n$ be a closed $k$-dimensional topological manifold. Let $G\subset\mathbb{R}^n$ be a $k$-dimensional definable superset of $\Gamma$. Then \[\Gamma\cap\overline{M_\Gamma}\subset G\cap \overline{M_G}.\]
\end{thm}
\begin{proof}
Choose a point $x\in\Gamma\cap\overline{M_\Gamma}$. The previous proposition allows us to approach $x$ with a sequence of points $g_\nu\in Reg_2 \Gamma$ with $r(g_\nu)$ convergent to zero. In such a case, by definition of the reaching radius, there exists a sequence of points $a_\nu\in M_\Gamma$ such that $g_\nu\in m(a_\nu)$ and $r(g_\nu)=\|g_\nu-a_\nu\|$. The last equality means, in particular, that the sequence $(a_\nu)$ converges to $x$. 

The dimension of the tangent cones $C_{g_\nu}G$ is bounded by $k$. Therefore, either $C_{g_\nu}G=C_{g_\nu}\Gamma$ and $a_\nu-g_\nu\in N_{g_\nu}G$ or $C_{g_\nu}G$ is not convex. In both cases, we can find elements of $M_G$ close to $x$.This first one, due to $M_G\cap [g_\nu,a_\nu]\neq \emptyset$ and convergence $[g_\nu,a_\nu]\to \lbrace x\rbrace$. In the second, due to the tangent cone criterion. 

\end{proof}
We are ready to show the final result on the non-marginal $\mathscr{C}^1$ singularities approached by the medial axis.
\begin{cor}\label{ostatni wniosek}
Assume that $X$ is a closed definable set, $x\in X$. Let $\Gamma \subset X$ be a topological manifold of dimension $\dim_x X$, and assume that $x\in \Gamma \backslash \overline{M_X}$. Then, there exists $U\in\mathcal{V}(x)$ such that $\Gamma\cap U= X\cap U$, and, in particular, $X$ is locally a $\mathscr{C}^{1,1}$ manifold at $x$.  
\end{cor}

\begin{proof}
Due to Theorem \ref{powiekszanie zbioru osiąganego przez szkielet}, $M_\Gamma$ is separated from $x$. Therefore, there exists $U$ - a neighbourhood of $x$, such that the multifunction $m_\Gamma$ is univalent in $U$. Moreover, the results from \cite{Rataj} assure that $\Gamma\cap U$ is a $\mathscr{C}^{1,1}$ manifold, meaning its tangent cones are $(\dim_x X)$-dimensional vector spaces. 

Assume now that points of $X\backslash \Gamma$ exist arbitrarily close to $x$. Surely $m_X(y)=\lbrace y\rbrace$ for $y\in X\cap U\backslash \Gamma$. We can see now that an intersection of $ \mathcal{N}(m_\Gamma(y))$ and $(y,m_\Gamma(y)]$ is nonempty, or $C_{m_\Gamma(y)}X$ is not convex. Either way, we obtain a contradiction with~$x\notin \overline{M_X}$.

The class of $X$ smoothness is now a consequence of the results from \cite{Rataj}.
\end{proof}

The last corollary states that the medial axis approaches all the $\mathscr{C}^1$-singular non-marginal points of the set. In particular, the medial axis must approach all singularities of a definable curve but its endpoints, where the issue depends on the superquadracity of the curve (cf. \cite{BirbrairDenkowski}). 

\begin{figure}
    \centering
    \includegraphics[width=0.7\textwidth,height=9cm]{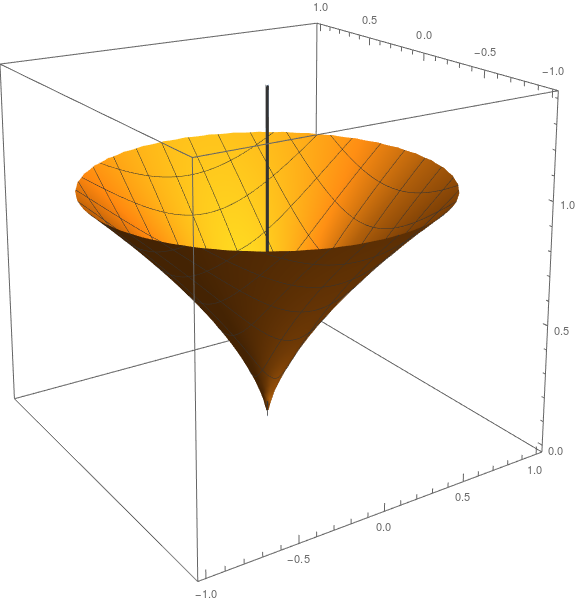}
    \caption{The set $X=\lbrace (z-\sqrt[3]{x^2+y^2})(x^2+y^2)=0\rbrace$ from Example~\ref{trabka}.}
    \label{fig:Trabka}
\end{figure}

\begin{ex}\label{trabka}
Regard $X:=\lbrace (z-\sqrt[3]{x^2+y^2})(x^2+y^2)=0\rbrace$. Neither the results of Rataj and Zaj\'i\v{c}ek nor Theorem~\ref{lepszy motzkin} settles whether $0\in \overline{M_X}$. Indeed the lower limit of the tangent cones equals the $z$-axis, and $X$ does not form a topological manifold for any neighbourhood of $0$. However, $\lbrace z-\sqrt[3]{x^2+y^2}=0\rbrace$ is a proper subset of $X$ and a $(\dim_0 X)$-dimensional manifold. Therefore we can derive that the origin belongs to $\overline{M_X}$ basing on Corollary~\ref{ostatni wniosek}
\end{ex}
\printbibliography 
\end{document}